\documentclass[preprint,3p,review,12pt]{elsarticle}

\usepackage{amssymb}
\usepackage{amsmath}
\usepackage{amsmath}

\usepackage{graphicx}

\usepackage{xcolor}

\usepackage{wasysym}

\usepackage{float}

\usepackage{utfsym}
\usepackage{amsthm}
\usepackage{thmtools}
\usepackage{lmodern}

\newtheorem{theorem}{Theorem}
\newtheorem{definition}{Definition}
\newtheorem{lemma}{Lemma}

\journal{Discrete Applied Mathematics}

\begin{document}

\begin{frontmatter}



\title{King Chasing Problem in Chinese Chess is NP-hard}


\author[iie,scs]{Chao Li} 
\ead{lichao1989@iie.ac.cn}
\affiliation[iie]{organization={Institute of Information Engineering, Chinese Academy of Sciences},
            city={Beijing},
            postcode={100085}, 
            country={China}}

\affiliation[scs]{organization={School of Cyberspace Security, University of Chinese Academy of Sciences},
            city={Beijing},
            postcode={100049}, 
            country={China}}

\author[gdmcfd]{Zhujun Zhang} 
\ead{zhangzhujun1988@163.com}
\affiliation[gdmcfd]{organization={Big Data Center of Fengxian District},
            city={Shanghai},
            postcode={201499}, 
            country={China}}

\author[gdufs]{Chao Yang\corref{cor1}}
\ead{yangchao@gdufs.edu.cn}
\cortext[cor1]{Corresponding author}
\affiliation[gdufs]{organization={School of Mathematics and Statistics, Guangdong University of Foreign Studies},
            city={Guangzhou},
            postcode={510006}, 
            country={China}}

\begin{abstract}
We prove that king chasing problem in Chinese Chess is NP-hard when generalized to n×n boards. `King chasing' is a frequently-used strategy in Chinese Chess, which means that the player has to continuously check the opponent in every move until finally checkmating the opponent's king. The problem is to determine which player has a winning strategy in generalized Chinese Chess, under the constraints of king chasing. Obviously, it is a sub-problem of generalized Chinese Chess problem. We prove that king chasing problem in Chinese Chess is NP-hard by reducing from the classic NP-complete problem 3-SAT.
\end{abstract}




\end{frontmatter}



\section{Introduction}
\label{sec:introduction}
Since the 1980s, the computational complexity of board games has been extensively studied by scholars. The computational complexity of many classic board games problem were determined. In 1980,  Lichtenstein and Sipser proved that planar Generalized Geography (GG) is PSPACE-complete by reducing the True Quantified Boolean Formula (TQBF) problem to it. Then they reduced GG to Go, demonstrating that Go is PSPACE-hard \cite{Lichtenstein1980}. In 1981,  Fraenkel and  Lichtenstein proved that Chess is EXPTIME-complete by reducing $ G_{3}$ (a Boolean game) to Chess \cite{Fraenkel1981}. In 2022,  Brunner et al. demonstrated that `retrograde' problem and `helpmate' problem in Chess are PSPACE-complete by reducing Subway Shuffle to these two problems \cite{Brunner2020}. In 1980,  Reisch proved that Gobang is PSPACE-complete \cite{Reisch1980}. In 1994, Iwata and  Kasai showed that Othello is PSPACE-complete by reducing Geography (on bipartite graphs with maximum degree 3) to the Othello game \cite{Iwata1994}. In 2024,  Burke and  Tennenhouse proved that Forced-Capture Hnefatafl is PSPACE-hard by reducing Positive Conjunctive Normal Form (Positive-CNF) to the game \cite{Burke2024}.

Chinese Chess is a classic two-player chess-like game, and it is said to have originated in China's Warring States Period (about 2,000 years ago)\cite{ZhouChenliang2024}. Chinese Chess is now one of the most popular chess-like games in China, Vietnam and other countries. King chasing in Chinese Chess is that player has to continuously check the opponent in every move until finally checkmating the opponent's king. King chasing is a frequently-used strategy in Chinese Chess. In Shogi (Japanese Chess), king chasing problem is called Tsume-shogi, which is the most popular genre of problem in Shogi. The king chasing problem in Chinese Chess can be stated more formally as below.

\begin{definition}[The king chasing problem in Chinese Chess]\label{king chasing problem}\ \par

\noindent {\bf Input.} A Chinese Chess game position that the player(take Red as an example) has to continuously check the opponent in every move.

\noindent {\bf Output.} Can Red wins?
\end{definition}

Obviously, The king chasing problem in Chinese Chess is a sub-problem of generalized Chinese Chess problem. From the perspective of computational complexity, the sub-problem is at least not more difficult. Therefore, even if Zhang has proved Chinese Chess problem to be EXPTIME-complete \cite{ZhangZhujun2019}, it does not indicate the computational complexity of king chasing problem in Chinese Chess.

The main result of this paper is the following theorem.

\begin{theorem}\label{thm_main}
  The king chasing problem in Chinese Chess is NP-hard.
\end{theorem}

We shall prove Theorem \ref{thm_main} in the section \ref{sec:the reduction} by reduction from the 3-SAT which is a classic NP-complete problem. NP-hardness of many problems and games are proven by reducing 3-SAT to them \cite{Aloupis2015}.

\begin{definition}[3-SAT problem]\ \par

\noindent {\bf Input.} A 3-CNF (conjunctive normal form) formula that a Boolean formula is in conjunctive normal form, and all the clauses have three literals.

\noindent {\bf Output.} Is there an assignment to the variables such that the formula is true?
\end{definition}

\begin{theorem}[\cite{Introduction-to-the-Theory-of-Computation}]\label{thm_sat}
  3-SAT is NP-complete.
\end{theorem}

The section \ref{sec:Rules of Chinese Chess} of this paper primarily introduces the rules of Chinese Chess, section \ref{sec:the reduction} provides a detailed description of the constructed gadgets and the reduction, section \ref{sec:conclusion} introduces the future research work.

\section{Rules of Chinese Chess}
\label{sec:Rules of Chinese Chess}
Chinese Chess has two players, Red and Black. The initial positions of the two players' pieces on the board are shown in Figure \ref{Fig.1}. The Chinese Chess board is comprised of several vertical and horizontal lines. The pieces are to be moved and played on the intersections \cite{World-Xiangqi-Rules2018}. There are two special areas on the board, called \textit{palace}, as the red dashed boxes shown in Figure \ref{Fig.1}. Each palace contains two oblique lines. The palace at the bottom of the board is red's palace, the other is black's palace.

In Chinese Chess game rules, \textit{check} is a move that attacks the opponent's king with the intent of capturing the king on the next move \cite{World-Xiangqi-Rules2018}. A king so threatened is said to \textit{be in check}. For example, if Red moves her rook to attack black king, usually called `red rook check Black', `Red check Black' or `check'. The player being in check must immediately take action to resolve the attack on its king, which is called \textit{resolve check} \cite{World-Xiangqi-Rules2018}. If a player cannot resolve check, meaning its king will inevitably be captured by the opponent on the next move, which is called \textit{checkmate}. The player that first checkmates the opponent's king wins the game. Taking Red as an example, king chasing refers to Red must check Black in every move until it checkmates black king, while Black continuously resolves check. 
\begin{figure}[H]
\begin{minipage}[t]{0.48\textwidth}
\centering
\includegraphics[width=\textwidth]{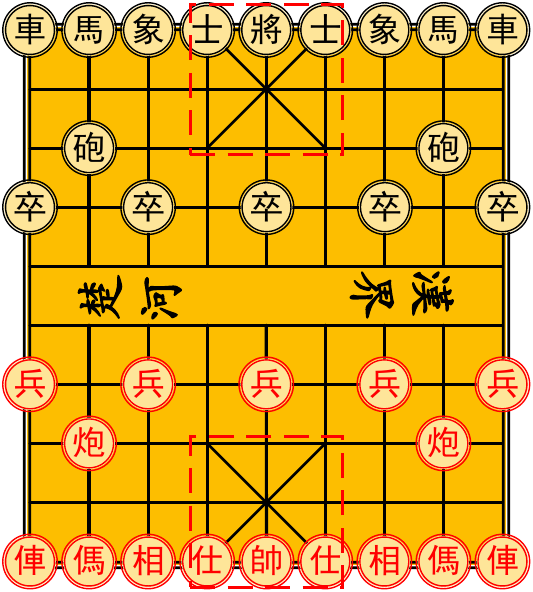}
\caption{The Chinese Chess board.}
\label{Fig.1}
\end{minipage}
\hfill
\begin{minipage}[t]{0.48\textwidth}
\centering
\includegraphics[width=\textwidth]{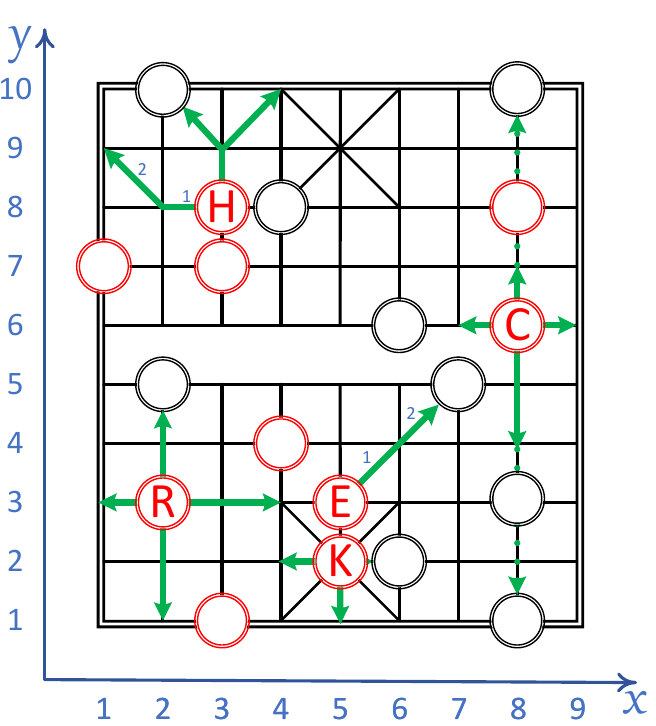}
\caption{The movement rules, intersections indicated by green arrows represent reachable intersections.}
\label{Fig.2}
\end{minipage}
\end{figure}

At the start of the game, Red moves first, and both players alternate making a move until a win, or draw is determined. When a player moves a piece from one intersection to another, it constitutes \textit{a move}. Moving a piece  from one intersection to several adjacent intersections is called \textit{move a square}. A piece moving along vertical line, horizontal line and diagonal line are respectively called \textit{a vertical move}, \textit{a horizontal move}, and \textit{a diagonal move}.

The next few paragraphs describe the move and capture rules for each type of piece in Chinese Chess. In Figure \ref{Fig.2}, intersections indicated by green arrows represent reachable intersections.

King: The king can move one square per move either vertically or horizontally. In Figure \ref{Fig.2}, red king at (5,2) can move to (4,2), (5,1), or (6,2) with a move. It can capture the black piece at (6,2) and occupy that intersection. However, it cannot move to (5,3) because that intersection is occupied by a red piece. Red king is denoted as \raisebox{-0.2ex}{\includegraphics[height=1em]{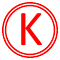}}, and black king as \raisebox{-0.2ex}{\includegraphics[height=1em]{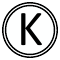}}.

Rook: The rook can move as many squares as the player wants in a vertical line or a horizontal line per move, as long as no piece blocks its path. In Figure \ref{Fig.2}, red rook at (2,3) can move to any intersection along the four green solid lines starting from it, such as (2,4), (1,3), etc., and it can also capture the black piece at (2,5) and occupy that intersection. Red rook is denoted as \raisebox{-0.2ex}{\includegraphics[height=1em]{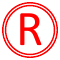}}, and black rook as \raisebox{-0.2ex}{\includegraphics[height=1em]{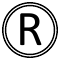}}.

Elephant: The elephant moves two squares diagonally in the same direction per move, and the direction of the two diagonal moves must remain consistent. If the intersection reached after the first diagonal move is occupied by any piece, it cannot move in that direction. In Figure \ref{Fig.2}, red elephant at (5,3) can move to (7,5), with the two diagonal moves being (5,3) to (6,4) and (6,4) to (7,5). However, it cannot move to (3,5) or (7,1) because (4,4) and (6,2) are occupied by pieces. It cannot move to (3,1) because that intersection is occupied by red piece. Red elephant is denoted as \raisebox{-0.2ex}{\includegraphics[height=1em]{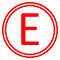}}. In gadgets constructed in this paper, there are numerous black elephants, and to highlight other pieces, black elephant is denoted as \raisebox{-0.2ex}{\includegraphics[height=1em]{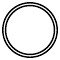}}. As we shall see in the next section, elephants are like unmovable walls in our construction of gadgets.

Horse: The horse moves one square vertically or horizontally followed by one square diagonally per move. If the intersection reached after the vertical or horizontal move is occupied by any piece, it cannot move in that direction. In Figure \ref{Fig.2}, red horse at (3,8) can move to (1,9), through first moving vertically one square from (3,8) to (2,8), second moving diagonally one square from (2,8) to (1,9). However, it cannot move to (5,9), (5,7), (2,6) and (4,6), because (4,8) and (3,7) are occupied by pieces. It cannot move to (1,7), because that intersection is occupied by red piece. Red horse is denoted as \raisebox{-0.2ex}{\includegraphics[height=1em]{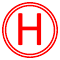}}, and black horse as \raisebox{-0.2ex}{\includegraphics[height=1em]{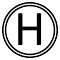}}.

Cannon: The movement of cannon is the same as rook. However, when a cannon captures an opponent's piece, it must leap over one piece. In Figure \ref{Fig.2}, red cannon at (8,6) can move to any intersection along the four green solid lines starting from it, such as (8,7), (7,6), (8,5), etc. Because there is one piece at (8,8), it can leap over the piece and capture the black piece at (8,10), and occupy that intersection. Similarly, it can also capture black piece at (8,1). Red cannon is denoted as \raisebox{-0.2ex}{\includegraphics[height=1em]{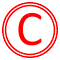}}, and black cannon as \raisebox{-0.2ex}{\includegraphics[height=1em]{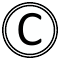}}.

This paper studies the generalized king chasing problem in Chinese Chess, that is, Definition \ref{king chasing problem} is studied in generalized Chinese Chess. In generalized Chinese Chess game position, there are one black king, one red king, and multiple other types of pieces on the board. The range of palace is generalized and expanded, meaning that generalized palace contains more intersections, naturally expanding the movement range of red king and black king.

Chinese Chess game contains seven types of pieces. But, in the reduction proof of this paper, only five types of pieces introduced above are used, and all pieces move within generalized palace. We will prove that, even with these restrictions, the king chasing problem in Chinese Chess is NP-hard.

\section{The reduction}
\label{sec:the reduction}
This paper proposes a protocol that constructs a game position for each formula in 3-SAT, such that the formula can be made `true' if and only if Red has a winning strategy. A game position is constructed by assembling gadgets that simulate the start, variable, and clause of 3-SAT. Throughout the game play, Red must use certain pieces to continuously check Black, while black king is forced to flee along a prearranged route to repeatedly resolve check. Whether Red can ultimately checkmate black king depends on whether the corresponding formula of 3-SAT can be made `true'.

In this paper, all of the constructed gadget locates on the generalized palace. In order to highlight pieces, diagonal lines on the board are not drawn.

\begin{figure}[H]
\begin{minipage}[t]{0.32\textwidth}
\centering
\includegraphics[width=\textwidth]{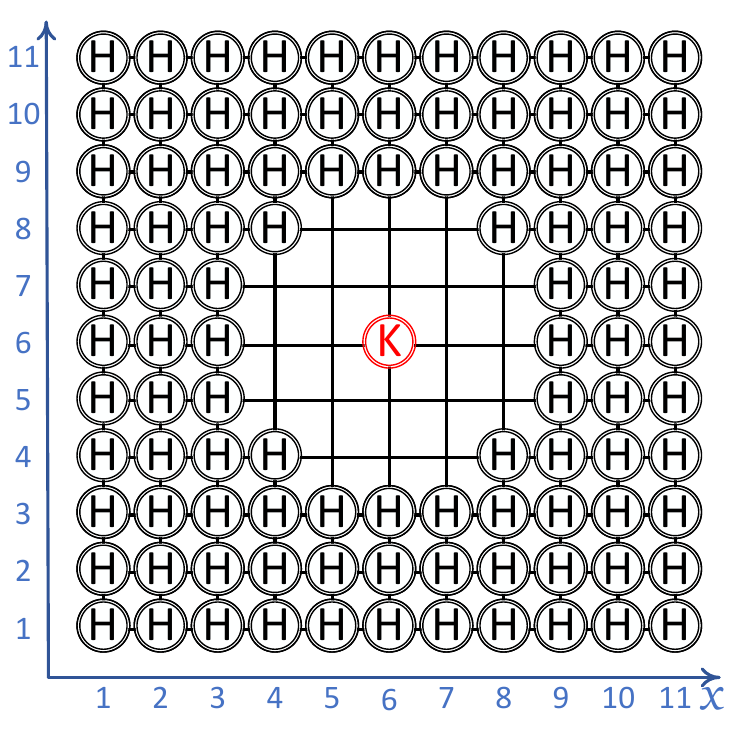}
\caption{The mate-in-one gadget.}
\label{Fig.3}
\end{minipage}
\hfill
\begin{minipage}[t]{0.32\textwidth}
\centering
\includegraphics[width=\textwidth]{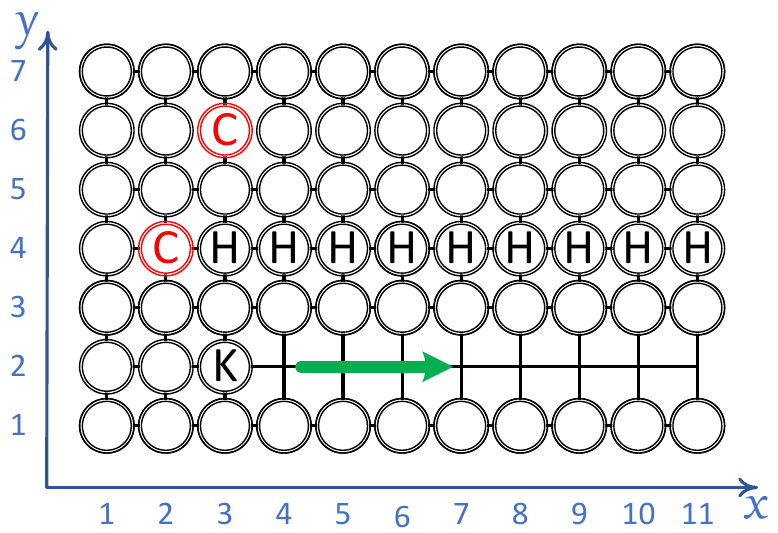}
\caption{The start gadget.}
\label{Fig.4}
\end{minipage}
\hfill
\begin{minipage}[t]{0.32\textwidth}
\centering
\includegraphics[width=\textwidth]{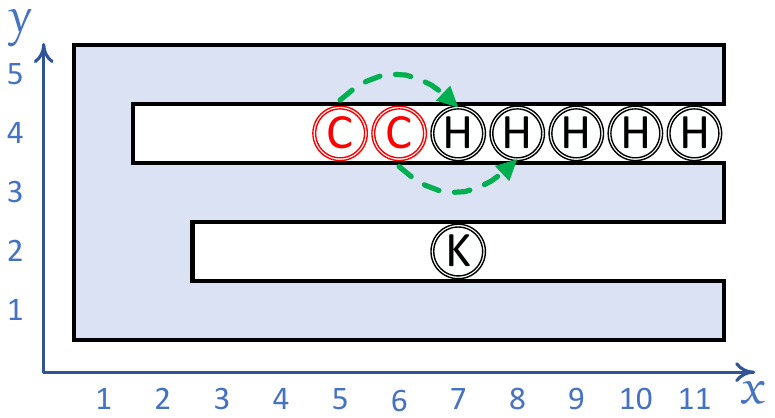}
\caption{Directional cannons alternately check Black.}
\label{Fig.5}
\end{minipage}
\end{figure}

\subsection{The mate-in-one gadget}
\label{subsec:the mate-in-one gadget}
Throughout the game play, Red needs to be in a state of being checkmated on the next move, thus we construct the mate-in-one gadget. In this gadget, red king is in a state which it will be checkmated by black horse on the next move. So Red has to check Black in every move, until Black is checkmated. As shown in Figure \ref{Fig.3}, a large number of black horses surround red king, these pieces are isolated in a specific area of the board, unaffected by the actions of other red pieces. If Red makes a move that does not check Black, Black will not need to resolve check. This means that black horse in this gadget can move and immediately checkmate red king, resulting in Red lose.

As shown in Figure \ref{Fig.3}, if black horse at (5,3) has the opportunity to move to (4,5), all possible positions that red king could move to would fail to resolve check. In other words, black horse checkmates red king. Therefore, to avoid being checkmated, Red must check Black in every move, forcing Black to resolve check in every move, this prevents Black from having the opportunity to move the horse in this gadget.

\subsection{ The start gadget}
\label{subsec:the start gadget}
The main idea of constructing start gadget is that Red continuously checks Black in every move, forcing black king to move unidirectionally along a specific direction. Here, \textit{unidirectional} means black king cannot move back to a position it has already passed through. There are two red cannons in the gadget, called \textit{directional cannons}. They will alternately capture black horses to check Black. During the whole game play, different pieces will take the responsibility of being the directional cannons.

As shown in Figure \ref{Fig.4}, at the beginning, red cannon at (3,6) captures black horse at (3,4) to check Black, forcing black king to move right to (4,2) to resolve check. Afterwards, red cannon at (2,4) moves to (4,4). At this point, these two red cannons form directional cannons. Thereafter, black king can only move right to resolve check. It cannot move left back to (3,2), as red cannon at (3,4) is attacking this position. In the gadget, Black cannot move his elephants or horses because black king being in check.

The two directional cannons alternately capture black horses on the same horizontal line to check Black, forcing black king to move unidirectionally. As shown in Figure \ref{Fig.5}, one directional cannon at (5,4) captures black horse at (7,4) to check black king at (7,2), forcing black king to move to (8,2). Another directional cannon at (6,4) captures black horse at (8,4) to continue check black king at (8,2).

\subsection{The turn gadget}

The purpose of constructing the turn gadget is to force the black king to make a $90^{\circ}$ degree turn along its way. As shown in Figure \ref{Fig.6}, the gadget contains an entrance and an exit.

When black king reaches the entrance of this gadget, two directional cannons alternately capture black horses to force black king to move left along the channel until it reaches (6,6). Red cannon at (4,8) moves to (6,8) to check, black cannon at (6,10) moves to (6,8) to resolve check, red cannon at (4,10) moves to (4,6) to continue checking Black, forcing black king to move downward to (6,5). Then, red cannon at (2,5) moves to (4,5). Red cannons at (4,6) and (4,5) become new directional cannons. Directional cannons alternately capture black horses, forcing black king to move downward (in the direction of the vertical green arrow in the figure). This changes the movement direction of black king.

There are also has other move sequences. When black king reaches (6,6), red cannon at (8,8) could move to (6,8) to check, and then black cannon at (6,10) could move to (6,8). In order to keep Black being in check, red cannon at (4,8) should move to (6,8), and black cannon at (6,11) moves to (6,8) to resolve check. The subsequent moves follow the same pattern as described above.

As shown in Figure \ref{Fig.6}. A black horse is placed at (8,7), instead of a black elephant. This is because, the elephant would be movable—for example, it could move to (6,5). More importantly, due to the mobility of the elephant, it would cause the reduction to fail. Suppose there is a black elephant at (8,7), when black king moves to (8,6), directional cannon at (10,8) moves to (8,8) to check Black, and elephant at (8,7) moves to (6,5) to resolve check (since a cannon must leap over one piece to capture). Red can then only move directional cannon at (8,8) to capture black elephant at (8,10) and check Black. The black king then continues to move left to (7,6), at which point Red cannot check black. After Red makes an arbitrary move, Black could move the horse in the mate-in-one-gadget, and immediately checkmating Red.

To prevent red cannons on the horizontal line ($y$=8) from trying to checkmate the black king, there are three black cannons placed on the vertical line ($x$=6). For example, when red cannon at (8,8) moves to (6,8), black cannon at (6,10) or (6,11) will move downward to capture the red cannon.

\begin{figure}[H]
\begin{minipage}[t]{0.48\textwidth}
\centering
\includegraphics[width=\textwidth]{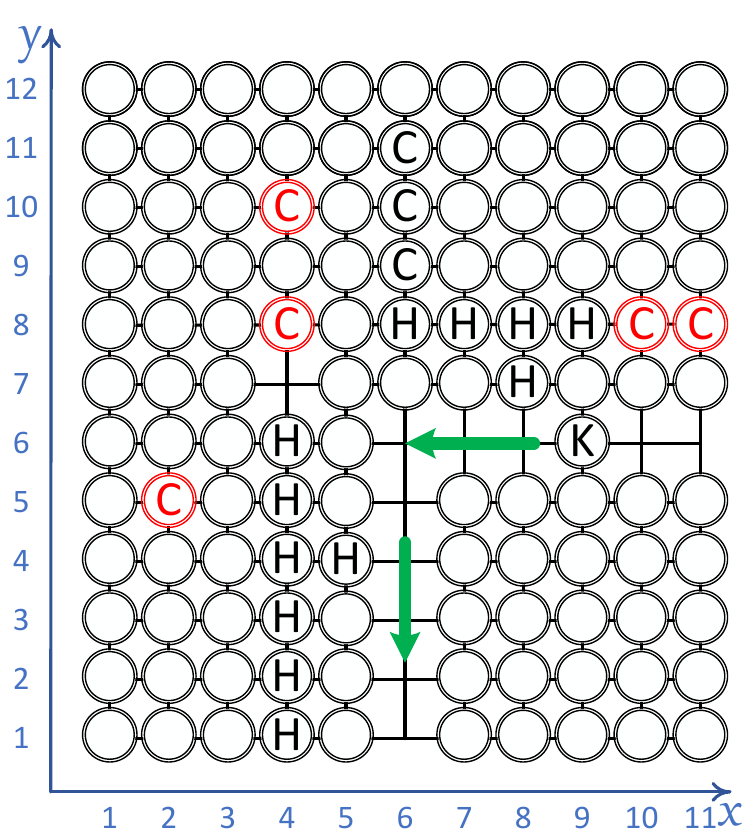}
\caption{ The turn gadget.}
\label{Fig.6}
\end{minipage}
\hfill
\begin{minipage}[t]{0.48\textwidth}
\centering
\includegraphics[width=\textwidth]{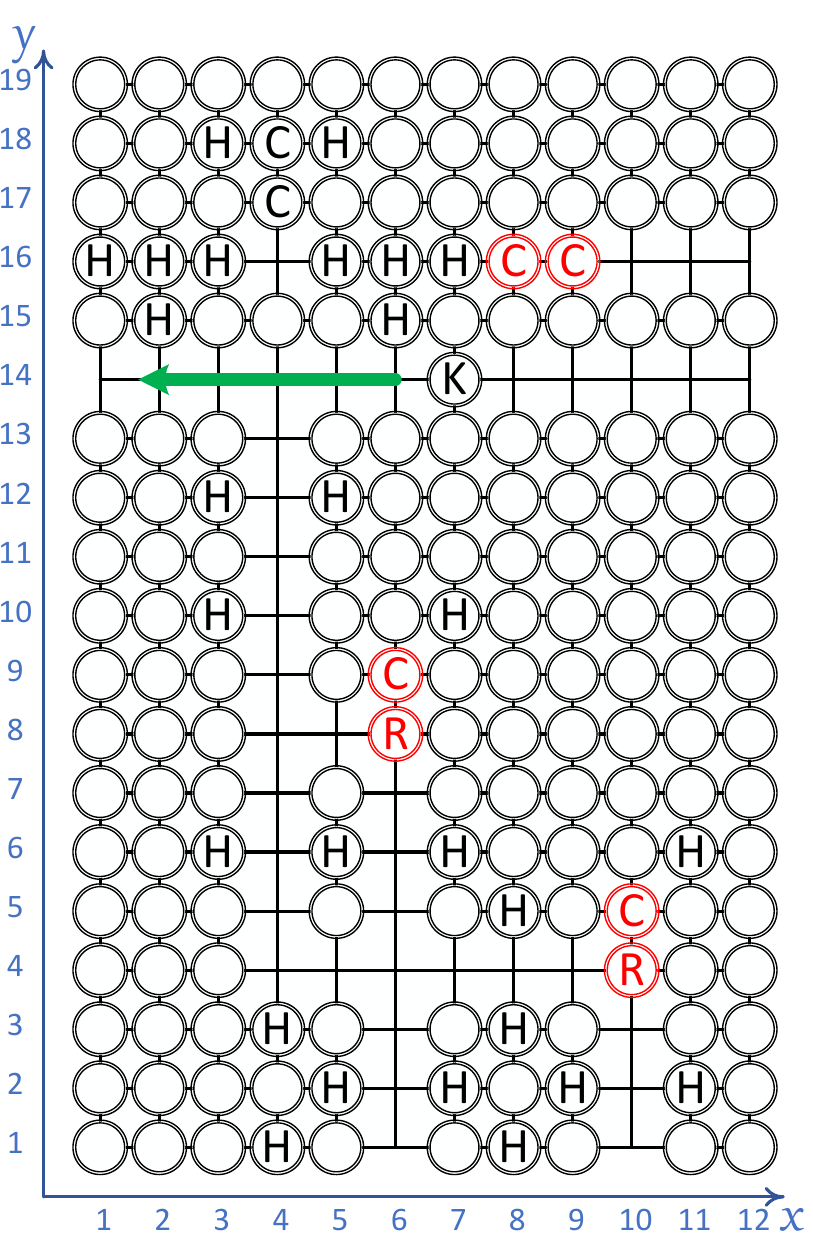}
\caption{The variable gadget.}
\label{Fig.7}
\end{minipage}
\end{figure}

\subsection{The variable gadget}
\label{subsec:the variable gadget}
A variable gadget, illustrated in Figure \ref{Fig.7}, simulates a variable in a formula. Black king enters from the entrance on the right and leaves through the exit on the left. There are two red rooks in the variable gadget. One of the two rooks should check Black, corresponding to the variable $x_i$ or its negation $\lnot x_i$ being chosen, such that once a rook is moved, the other rook can never be moved. The red cannons positioned above red rooks are called \textit{variable cannon}. The cannon on the left is corresponding to variable $x_i$, and another on the right is corresponding to $\lnot x_i$.  The red cannon at (6,9) represents $x_i$, cannon at (10,5) represents $\lnot x_i$. After a rook moves, the corresponding variable cannon can move downward to check Black, collaborating with other pieces complete the verification of the literals in clause gadgets (these will be explained in Subsection \ref{subsec:the logical connection between variable gadget and clause gadget}).

As shown in Figure \ref{Fig.7}, when the black king reaches to (4,14), Red cannot use directional cannons to check Black. This is because, if directional cannon at (5,16) moves to (4,16) to check Black, black cannon at (4,18) can capture that directional cannon to resolve check. In this case, Red would be left with only one directional cannon, unable to continue alternating check, which would cause Red to lose. So Red can only use one of the two rooks to check Black. If red rook at (6,8) moves to (4,8) to check, then it is corresponding to that variable $x_i$ is assigned true. If red rook at (10,4) moves to (4,4) to check, then it is corresponding to that variable $x_i$ is assigned false. Only one of the two red rooks can move to vertical line ($x$=4), simulating the assignment of a variable. Thereafter, black king continues moving left, and two directional cannons continue to alternately check black king.

When black king moves to (3,14), if Red uses the rook that moved to the vertical line ($x$=4) in the previous step to move upward to (4,14) in an attempt to checkmate black king, black cannon at (4,17) will immediately capture that red rook to resolve check. Then Red has to continue using directional cannons to check Black. In this case, black cannon will remain at (4,14), but this has no impact in the game.

\subsection{The clause gadget}
\label{subse:the clause gadget}
A clause gadget simulates a clause in a formula, illustrated in Figure \ref{Fig.8}. Black king enters from the entrance at the bottom and leaves through the exit at the top. There are three red cannons that aligned vertically with red rook and black rook, called \textit{literal cannon}. Each literal cannon corresponds to a literal in a clause, and three literal cannons simulate the disjunction of three literals.

\begin{figure}[H]
\centering
\begin{minipage}[t]{0.48\textwidth}
\centering
\includegraphics[width=\textwidth]{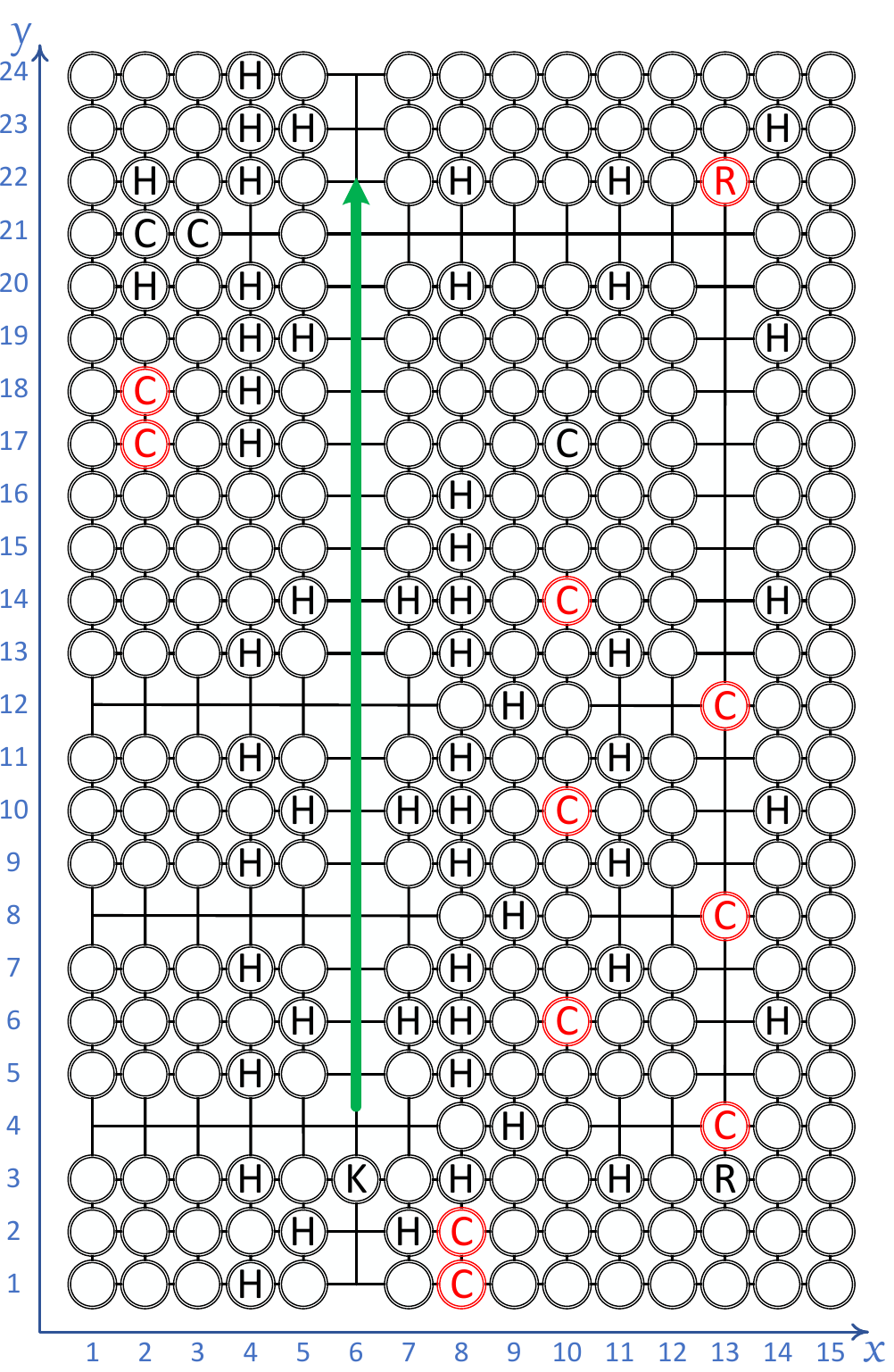}
\caption{ The clause gadget.}
\label{Fig.8}
\end{minipage}
\end{figure}

When black king reaches to the horizontal line occupied by a literal cannon, Red has only two methods to check Black. One is to use a literal cannon, this subsection will introduce in detail; and another is to use red rook on the left (not shown in Figure \ref{Fig.8}), this method will be explained in Subsection \ref{subsec:the logical connection between variable gadget and clause gadget} later.

Variable $x_i$ is assigned either true ($x_i$) or false ($\lnot x_i$), if variable assignment makes the corresponding literal true, the assignment is said to satisfy the literal. For example, if $x_i$ = \textbf{true}, the assignment of variable $x_i$ satisfies literal $x_i$ but does not satisfy literal $\lnot x_i$.

\begin{lemma}\label{lem_literal}
 If variable assignment satisfies literal, then in the constructed game position, the literal cannon corresponding to that literal will not be used to check Black; otherwise, it will move to check.
\end{lemma}

Lemma \ref{lem_literal} will become evident after the complete interconnection mechanism between the variable gadgets and the literal gadgets have been introduced in Subsection \ref{subsec:the logical connection between variable gadget and clause gadget} and Subsection \ref{subsec:the force-to-bottom gadget}. Provided Lemma \ref{lem_literal} holds, we now prove the following lemma on the property of clause gadget.

\begin{lemma}
  If variable assignments make the clause evaluate to true, then in the constructed game position,  black king can pass through clause gadget corresponding to that clause.
\end{lemma}

\begin{proof} We prove the lemma with clause $x_1 \vee \lnot x_2 \vee x_3$ as an example, and the argument applies to any clause in general. Obviously, this clause is satisfied by the variable assignments: $x_1$=\textbf{true}, $x_2$=\textbf{false}, $x_3$=\textbf{false}. As shown in Figure \ref{Fig.8}, red cannons at (13,12), (13,8) and (13,4) are the three literal cannons, corresponding to literals $x_1$, $\lnot x_2$ and $x_3$ respectively.

When black king is driven to (6,4), Red cannot use two directional cannons to check Black and must use either literal cannon at (13,4) or red rook on the left. Since the assignment of the variable $x_3$ is false, it does not satisfy literal $x_3$, so Red only uses literal cannon at (13,4) to capture black horse at (9,4) to check. When black king is driven to (6,8), the assignment of variable $x_2$ is false, which satisfies literal $\lnot x_2$, then Red uses the rook on the left of vertical line ($y$=8) to check (detailed in Subsection \ref{subsec:the logical connection between variable gadget and clause gadget}). After black king sequentially passes through the three literal cannons, their positions are shown in Figure \ref{Fig.9}.

As shown in Figure \ref{Fig.8}, when black king is driven to the horizontal line ($y$=21) occupied by black double cannons, Red can only use the rook above to move downward to check Black. At this time, the positions of three literal cannons can be in one of two scenarios. In the first scenario, all three literal cannons have moved away, as shown in Figure \ref{Fig.10}, and Black uses his rook to capture red rook to resolve check. In this case, Red cannot check Black anymore, leading to Black checkmate red king on the next move. In the second scenario, at least one literal cannon has not moved, and Black cannot use his rook to capture red rook to resolve check, forcing it to continue moving upward to resolve check, as shown in the Figure \ref{Fig.9}. In this case, black king successfully passes through the clause gadget. This simulates the condition in a formula where, as long as one of the three literals in a clause evaluates to true, the entire clause evaluates to true.

\begin{figure}[H]
\begin{minipage}[t]{0.523\textwidth}
\centering
\includegraphics[width=\textwidth]{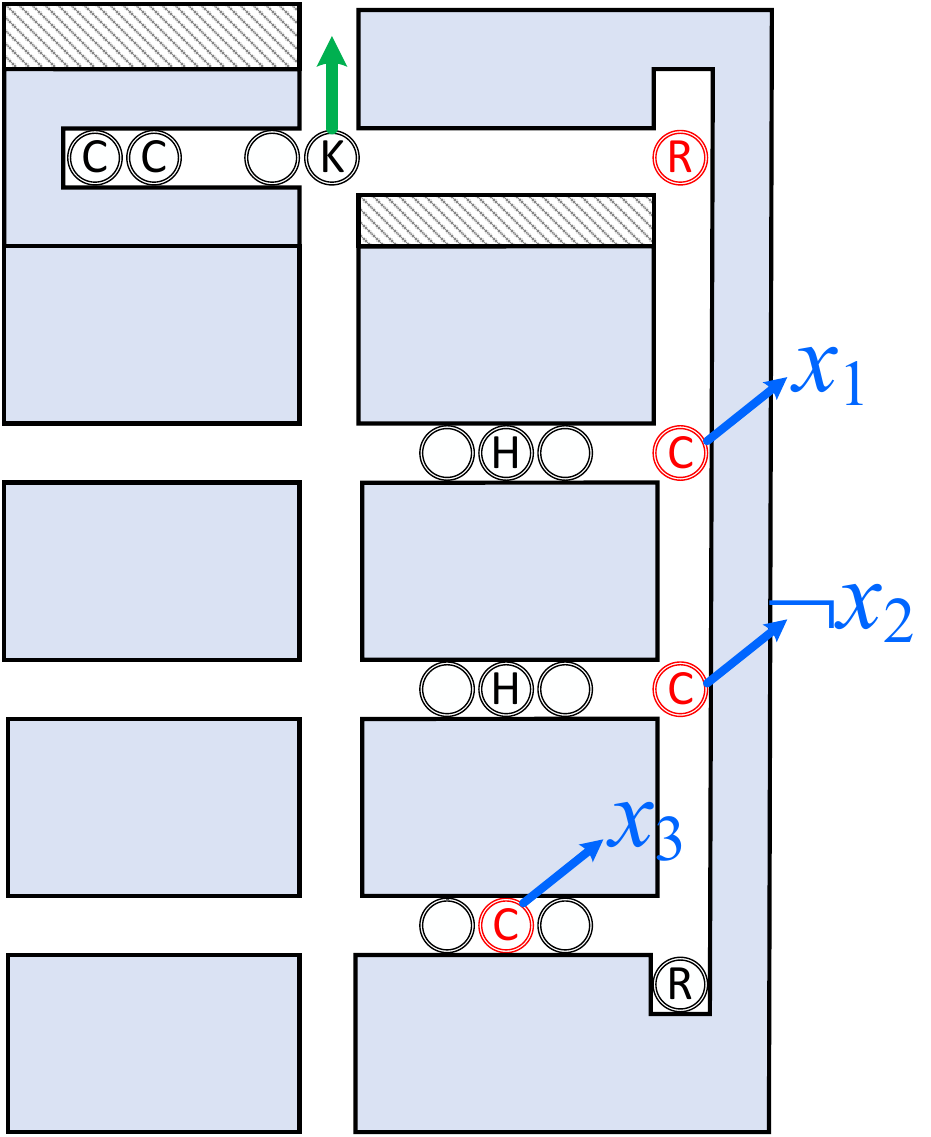}
\caption{Variable assignments make clause evaluate to true.}
\label{Fig.9}
\end{minipage}
\hfill
\begin{minipage}[t]{0.437\textwidth}
\centering
\includegraphics[width=\textwidth]{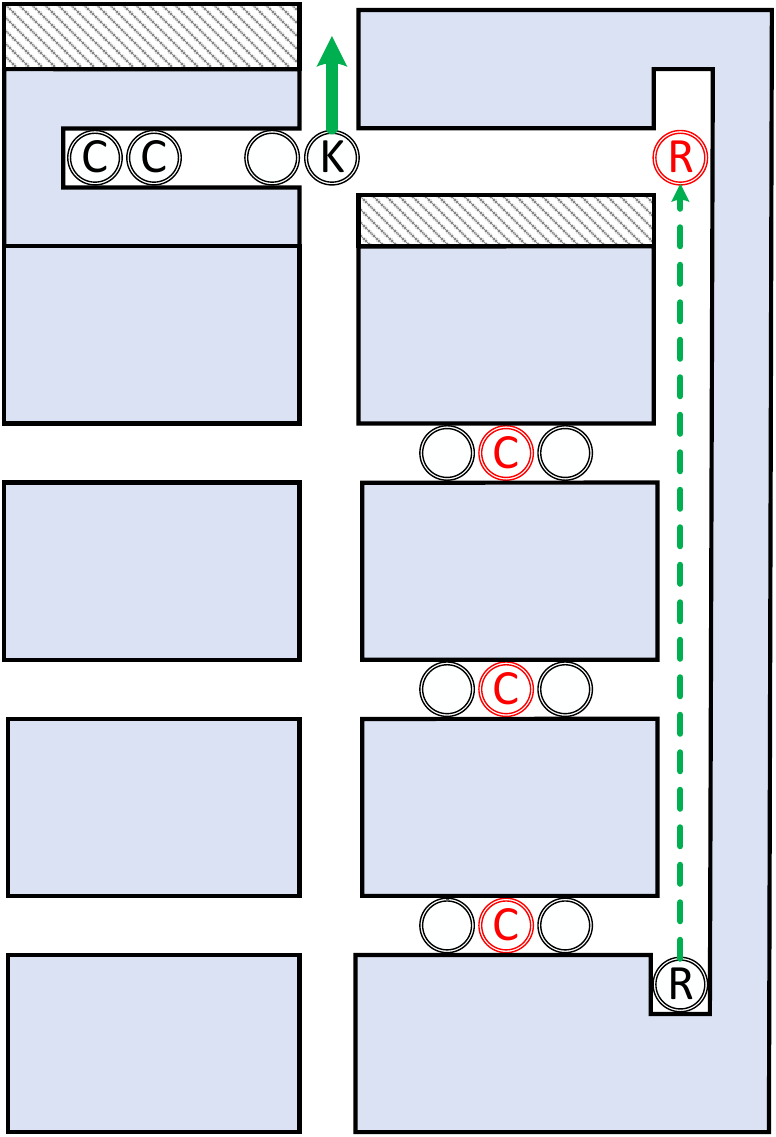}
\caption{Variable assignments make clause evaluate to false.}
\label{Fig.10}
\end{minipage}
\end{figure}

As shown in Figure \ref{Fig.8}, after black king passes through the three literal cannons, it reaches to (6,17). Due to black cannon at (10,17), Red cannot use directional cannons at (8,15) and (8,16) to check Black. The cannons at (2,17) and (2,18) become new directional cannons, which will force black king to move upward to (6,21). At this point, if Red moves directional cannon at (4,20) to (4,21) to check Black, black cannon at (2,21) can immediately capture that directional cannon to resolve check. Red has only one directional cannon, unable to continue alternating check Black, resulting in Red lose. Therefore, Red must use rook at (13,22) to move downward to (13,21) to check.

When black king moves to (6,22), if red rook at (13,21) moves to (6,21) in an attempt to checkmate black king, Black can use the cannon on the horizontal line ($y$=21) to capture that red rook to resolve check, forcing Red to continue using directional cannons to check. In this case, the black cannon will remain at (6,21), but this has no impact on the reduction.\end{proof}

\subsection{ The logical connection between variable gadget and clause gadget}
\label{subsec:the logical connection between variable gadget and clause gadget}
In the constructed game position, the most critical part is to establish a logical connection between the variable gadgets that simulate variables and the clause gadgets that simulate clauses. That is, the variable cannons cooperating with clause gadget to achieve clause verification. In this subsection, we take formula $\Phi$ $= (x_1 \vee \lnot x_2 \vee x_3)$ as an example to show the interconnection (see Figure \ref{Fig.11}) between the variable gadgets and the clause gadgets. The formula $\Phi$ is true by setting the variables $x_1$=\textbf{true}, $x_2$=\textbf{false} and $x_3$=\textbf{false}.

The connections are built by placing black rooks at certain locations in the following way. If a clause includes literal $x_i$ in the formula, then a black rook will be placed at the intersection of the horizontal line occupied by its corresponding literal cannon and the vertical line occupied by variable cannon-$x_i$. If a clause includes literal $\lnot x_i$, then a black rook will be placed at the intersection of the horizontal line occupied by its corresponding literal cannon and the vertical line occupied by variable cannon-$\lnot x_i$. These black rooks that are placed at intersection locations are called \textit{key rook}. The key rooks and literal cannons come in pairs, so there are the same number of key rooks and literal cannons. For example, a key rook at (7,12) is shown in Figure \ref{Fig.11}.

After the black king passing through all variable gadgets and completing variable assignments, it takes a $90^{\circ}$ degree turn and moves downward. When black king moves to the same horizontal line as a literal cannon, if the variable assignment satisfies that literal, the variable cannon moves downward to capture key rook to check, thereby releasing the red rook located to the lower left of the key rook. When black king moves upward later through the clause gadget and is again on the same horizontal line as that literal cannon, the released red rook can move upward to check Black, eliminating the need to move the literal cannon. A movable variable cannon will capture all key rooks on the same vertical line by itself.

The variable gadgets and clause gadgets interact in two situations as shown in Figure \ref{Fig.11}. In the first situation, variable assignment satisfies literal in the clause. For example, variable $x_1$=\textbf{true}, it satisfies the literal $x_1$ in the clause. When black king moves to (2,12), variable cannon-$x_1$ at (7,21) moves downward to capture the key rook at (7,12) to check Black. Then the variable cannon will be forced to move down due to the force-to-bottom gadget, which will be introduced in Subsection \ref{subsec:the force-to-bottom gadget}. When black king moves to (34,12), red rook at (5,11) could move upward to (5,12) to check Black, without needing to move the literal cannon at (40,12). In the second situation, variable assignment does not satisfy literal in clause. For example, variable $x_3$=\textbf{false}, it does not satisfy literal $\lnot x_3$. When black king moves to (2,4), Red can only use directional cannon to check Black, and red rook at (5,3) is not released. When black king moves to (34,4), Red has to move literal cannon at (40,4) to the left to check Black.

\begin{figure}[H]
\begin{minipage}[t]{\textwidth}
\centering
\includegraphics[width=\textwidth]{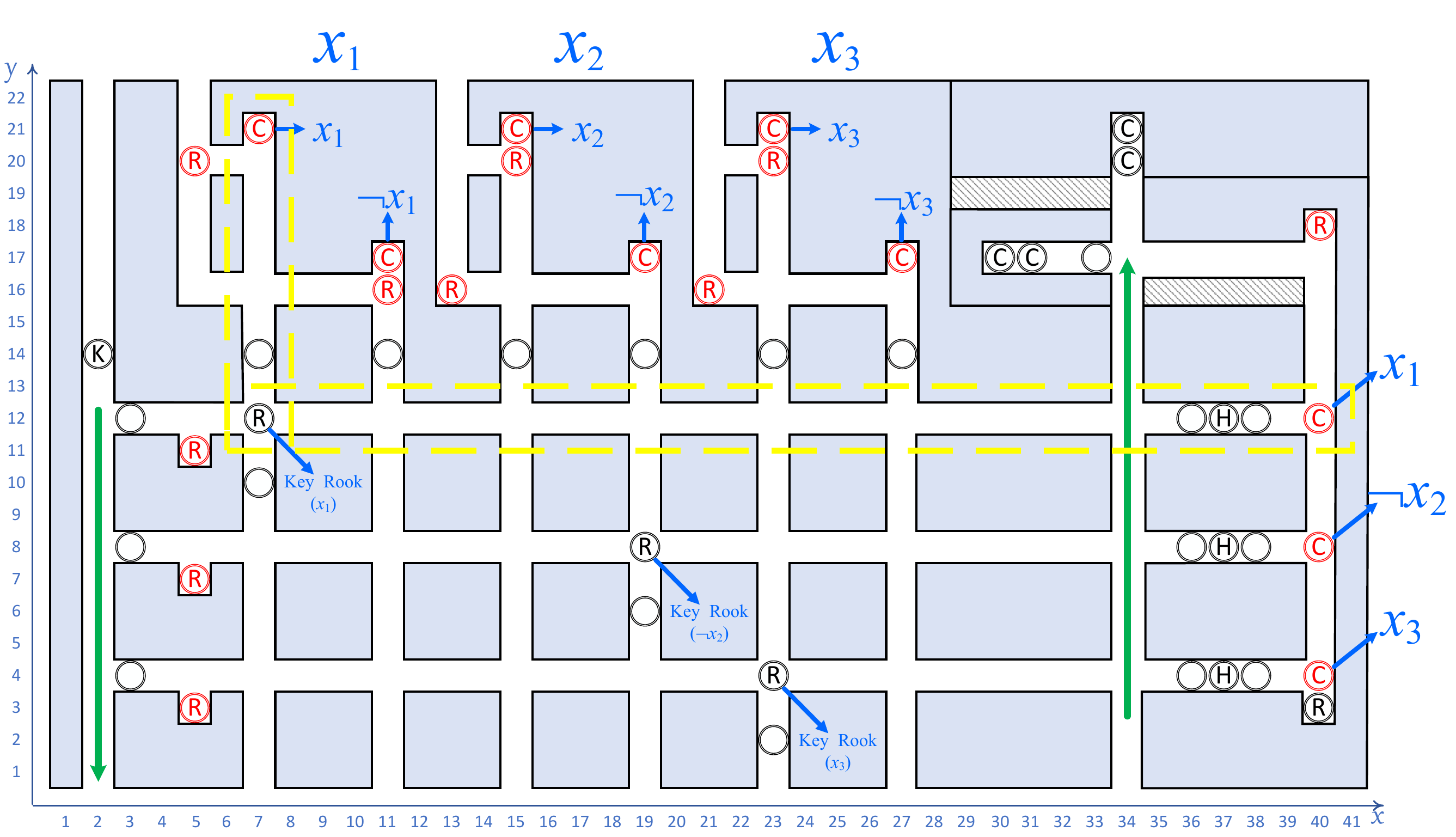}
\caption{Variable cannons cooperate with the clause gadget to achieve clause verification.}
\label{Fig.11}
\end{minipage}
\end{figure}

It should be noted, during the upward movement of black king, the red rooks on the left and right of the channel ($x$=34) in an attempt to move into the channel to checkmate Black, there are two corresponding methods of resolving check. If the right red rook in the clause gadget moves into the channel, the black cannon on the horizontal line ($y$=17) will capture the red rook to resolve check, this situation has been described in detail in Subsection \ref{subse:the clause gadget}. If the left red rook moves into the channel, the black cannon at the top of the channel will capture the rook. For instance, when black king moves to (34,13), if red rook at (5,12) moves right to (34,12) in an attempt to checkmate black king, black cannon at (34,20) will capture that red rook. If multiple red rooks move right into the upward channel in an attempt to checkmate black king, black cannon at the top of the channel will sequentially capture the rooks to resolve check, as shown in Figure \ref{Fig.12}.

\begin{figure}[H]
\centering
\begin{minipage}[t]{0.9\textwidth}
\centering
\includegraphics[width=\textwidth]{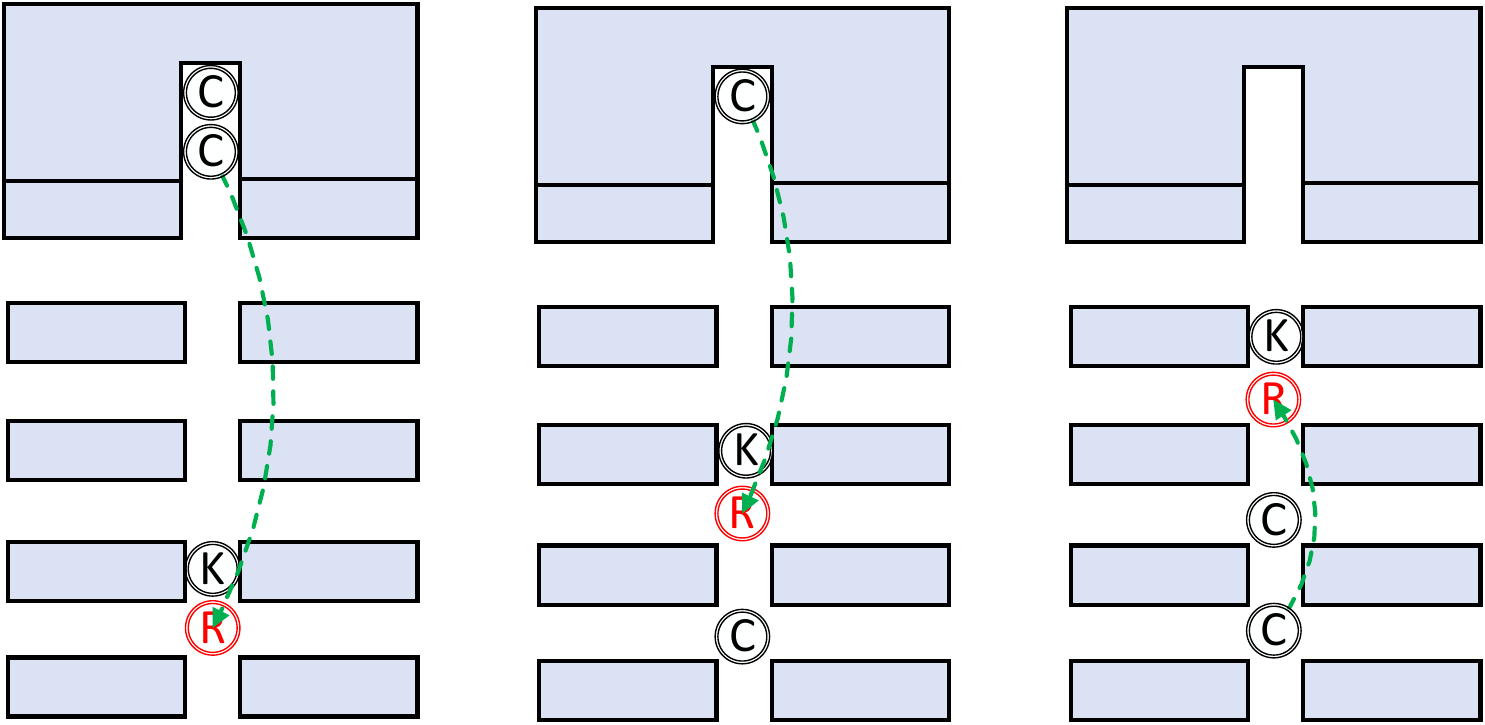}
\caption{ Black cannon at the top of the channel sequentially captures rooks that from the left into the channel to resolve check.}
\label{Fig.12}
\end{minipage}
\end{figure}

\subsection{The force-to-bottom gadget}
\label{subsec:the force-to-bottom gadget}
When black king moves upward, if variable cannons are not at the bottom, Red can deliver certain unintended moves, enabling the existence of a winning strategy even when the variable assignments do not satisfy the formula.

The variables, clauses, and variable assignments used in this section are the same as those in Subsection \ref{subsec:the logical connection between variable gadget and clause gadget}. As shown in Figure \ref{Fig.13}, when black king moves right to (19,2), Red unintended avoids using variable cannon-$\lnot x_2$ at (19,11) and instead uses a directional cannon to check Black. When black king moves upward to (34,11), Red unintended uses literal cannon ($\lnot x_2$) to check Black, despite the assignment of variable $x_2$ satisfying literal $\lnot x_2$. As black king continues moving upward to (34,17), key rook at (5,15) moves right into the channel ($x$=34) in an attempt to checkmate black king, black cannon at (34,23) above the channel captures the key rook to resolve check. At this point, variable cannon-$\lnot x_2$ at (19,11) moves right into the channel to checkmate black king. Then  Red wins, although black king had not yet completed the constructed game position. Seriously, in this case, even if variables assignment don’t satisfy the formula, Red may still win. For example,  $\Phi$ $= (x_1 \vee \lnot x_2 \vee x_3)\land(\lnot x_1\vee x_2\vee  x_3)$, variable assignments are the same as above,  in the game position is constructed based on formula $\Phi$ ,Red can deliver the same unintended moves to win.
\begin{figure}[H]
\centering
\begin{minipage}[t]{0.9\textwidth}
\centering
\includegraphics[width=\textwidth]{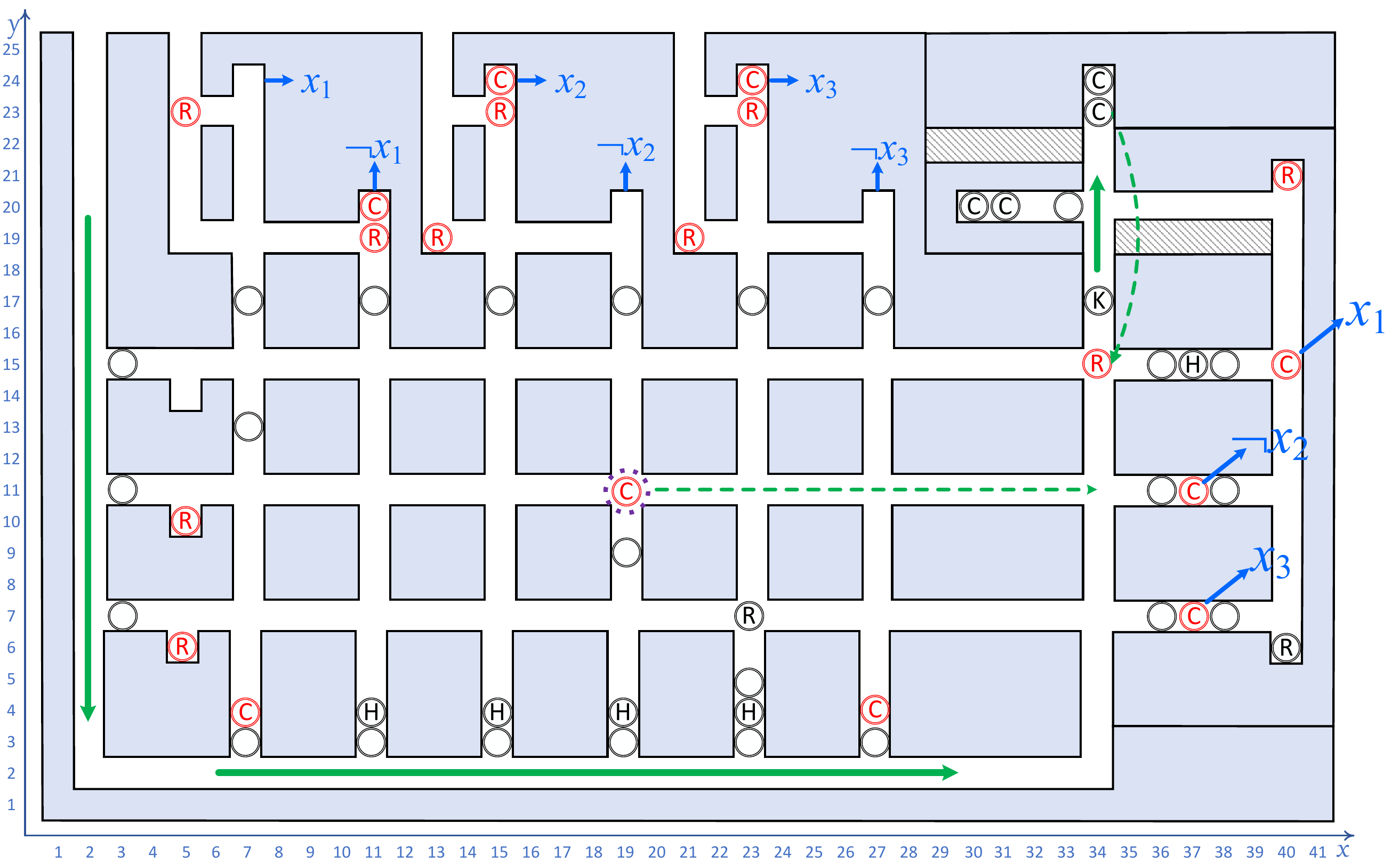}
\caption{Red delivers unintended moves.}
\label{Fig.13}
\end{minipage}
\end{figure}

So we set force-to-bottom gadgets. The purpose of the gadget is to compel variable cannons to ultimately move to the bottom of the constructed game position. As shown in Figure \ref{Fig.14}, the area within the red dashed box is the force-to-bottom gadget. When black king moves directly below a variable cannon, Red cannot continue to use directional cannons to check Black. If this variable cannon is free to move, Red must use it to move downward to deliver a check; if the variable cannon is not free to move, Red uses the cannon below it to move left or right to check. For example, when black king moves to (4,7), Red must use variable cannon-$x_i$ at (4,21) to move downward to (4,10) to check Black. When black king moves to (6,7), red cannon at (6,3) moves to (6,5) to check and replaces the cannon at (2,5) as the new directional cannon. When black king moves to (8,7), Red can only use the cannon at (6,2) to move right to (8,2) to deliver a check. When black king moves to (10,7), red cannon at (10,3) moves to (10,5) to check and becomes the new directional cannon.

\begin{figure}[H]
\centering
\begin{minipage}[t]{0.3\textwidth}
\centering
\includegraphics[width=\textwidth]{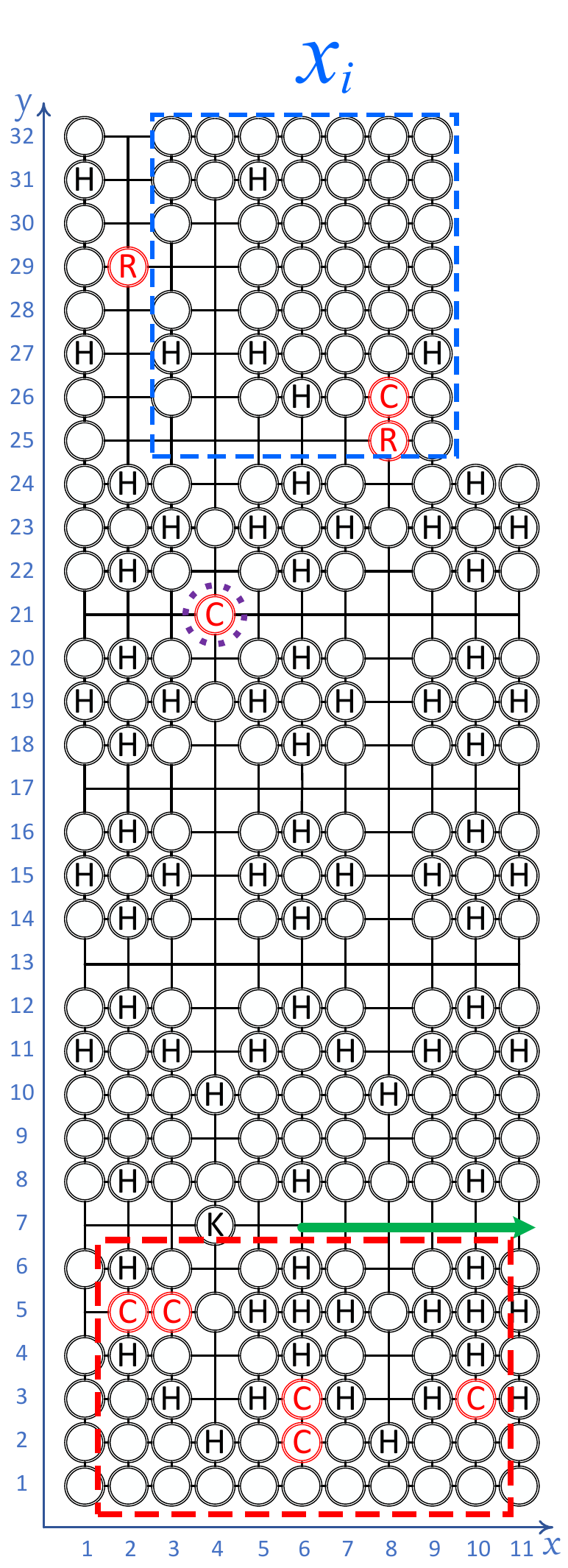}
\caption{The force-to-bottom gadget.}
\label{Fig.14}
\end{minipage}
\end{figure}

\subsection{Piecing together}
\label{subsec:piecing together}
We construct a game position for each formula in 3-SAT, the game position is composed of the mate-in-one gadget, the start gadget, the turn gadget, the variable gadget, the force-to-bottom gadget and the clause gadget, etc. For example, the formula $\Phi$ $=( x_1 \vee \lnot x_2 \vee x_3)\land(x_1\vee x_2\vee \lnot x_3)$. As shown in Figure \ref{Fig.15}, the game position is constructed based on formula $\Phi$. The two orange dashed boxes represent the two clauses, where the shorter box corresponds to clause $( x_1 \vee \lnot x_2 \vee x_3)$, the longer box corresponds to clause $(x_1\vee x_2\vee \lnot x_3)$.

\begin{lemma}
  \label{lemma-3}
   If there exists an assignment of values to a set of variables that satisfies the formula, then Red can continue checking Black in every move until finally checkmating black king, Red wins.
\end{lemma}

\begin{proof}
We proof using the above-mentioned formula $\Phi$ and the variable assignments $x_1$=\textbf{true}, $x_2$=\textbf{true}, $x_3$=\textbf{false} as an example. After the game begins, two red directional cannons drive black king to move leftward, sequentially passing through three variable gadgets. Based on the assignment of three variables, red rooks at (33, 37), (21, 41), (13, 41) move leftward in turn to check Black. Upon completing the variable assignment, black king passes through a turn gadget and shifts to move downward.

During black king downward movement, variable cannon-$x_1$ at (13, 42) captures downward key rook at (13, 33), (13, 21); variable cannon-$x_2$ at (21, 42) captures downward key rook at (21, 17); and variable cannon-$\lnot x_3$ at (33, 38) captures downward key rook at (33, 13). These key rooks are marked with purple dashed circles in Figure \ref{Fig.15}. Subsequently, black king passes through a turn gadget and shifts to move rightward.

During black king rightward movement, it sequentially passes through three force-to-bottom gadgets. Variable cannon-$x_1$ at (13, 21) moves to (13, 10), variable cannon-$x_2$ at (21, 17) moves to (21, 10), and variable cannon-$\lnot x_3$ at (33, 13) moves to (33, 10). Additionally, red cannon at (15, 2) moves to (17, 2), red cannon at (23, 2) moves to at (25, 2), and red cannon at (31, 2) moves to (29, 2). Subsequently, black king passes through a turn gadget and shifts to move upward.

\begin{figure}[H]
\begin{minipage}[t]{\textwidth}
\centering
\includegraphics[width=\textwidth]{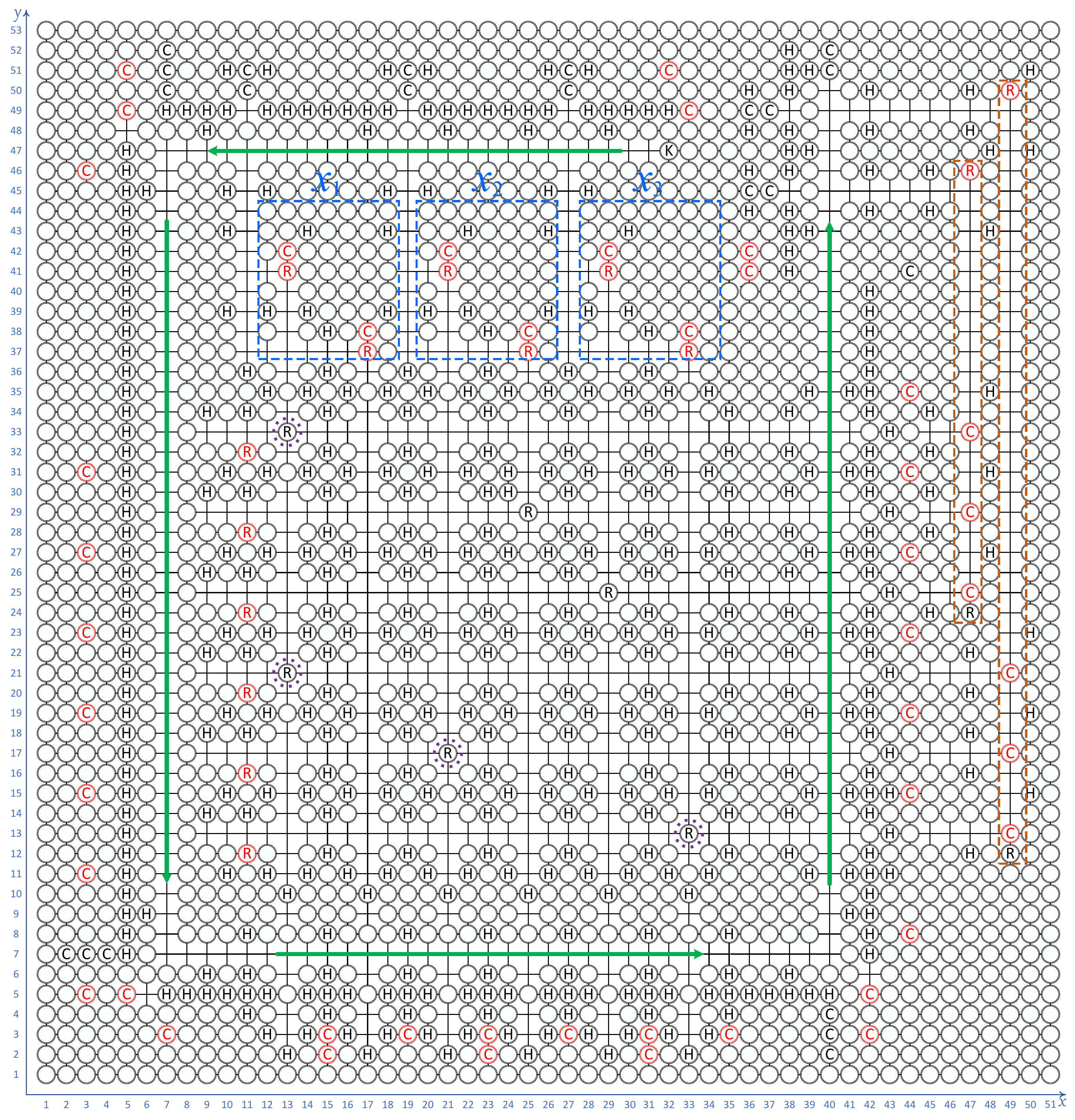}
\caption{The piecing together.}
\label{Fig.15}
\end{minipage}
\end{figure}

During black king upward movement, it sequentially passes through two clause gadgets. When black king passes through clause gadget corresponding to $(x_1\vee x_2\vee \lnot x_3)$, red rooks at (11, 12), (11, 16), (11, 20) move upward and check Black in sequence, while literal cannons at (49, 13), (49, 17), (49, 21) does not move. The black king passes through smoothly. When black king passes through clause gadget corresponding to $( x_1 \vee \lnot x_2 \vee x_3)$, only the literal cannon at (47, 33) is not used to check Black, so black king also passes through smoothly.

Finally, the unidirectional tunnel for the black king comes to an end after leaving the clause gadgets, so Red checkmates Black by with directional cannons.
\end{proof}

\begin{lemma}
  \label{lemma-4}
  If Red can continue checking Black in every move until finally checkmating black king, then there exists an assignment of values to the variables that satisfies the formula.
\end{lemma}

\begin{proof}
If Red continues checking black king, forcing black king to move downward, leftward, and upward along the path, ultimately checkmating black king, then there exists a set of assignments that can satisfy the formula. When black king passes through a variable gadget, if Red uses the left rook to check, the corresponding variable is assigned true; if Red uses the right rook to check, the corresponding variable is assigned false. When black king passes through a clause gadget, if a literal cannon is not used to check, the corresponding literal is true; if a literal cannon is used to check, the corresponding literal is false. In a clause gadget, as long as one of the three literal cannons is not used, black rook cannot capture red rook, so the assignment satisfies the corresponding clause. Thus, the formula is satisfied.
\end{proof}

\begin{proof}[Proof of Theorem \ref{thm_main}]
By combining Lemma \ref{lemma-3} and Lemma \ref{lemma-4}, a 3-CNF formula is satisfiable if and only if Red can win the corresponding position of the Chinese Chess game by chasing the king. By the NP-completeness of the $3$-SAT problem (Theorem \ref{thm_sat}), the king chasing problem in Chinese Chess is NP-hard. This completes the proof.
\end{proof}

\section{Conclusion}
\label{sec:conclusion}
We investigate the computational complexity of king chasing problem in Chinese Chess. We reduce the classical 3-SAT problem to the problem and proving that king chasing problem in Chinese Chess is NP-hard. Since the problem is obviously in EXPTIME, a natural next step for future work is to prove PSPACE-hardness or even EXPTIME-hardness of king chasing problem in Chinese Chess.

\bibliographystyle{elsarticle-num.bst}
\bibliography{refs}

\end{document}